\documentclass[12pt]{elsarticle}

\usepackage{amsmath}
\usepackage{amssymb}
\usepackage{fancybox}
\usepackage{eucal}
\usepackage{amsthm}
\usepackage{amscd}
\usepackage{hyperref}
\usepackage{wasysym}
\usepackage{url}

\usepackage{amssymb,}
\usepackage[]{amsmath, amsthm, amsfonts,graphicx, amscd,}
\usepackage[all,cmtip]{xy}
\input amssym.def \input amssym
\begin{document}

\newtheorem {thm}{Theorem}[section]
\newtheorem{corr}[thm]{Corollary}
\newtheorem {alg}[thm]{Algorithm}
\newtheorem*{thmstar}{Theorem}
\newtheorem{prop}[thm]{Proposition}
\newtheorem*{propstar}{Proposition}
\newtheorem {lem}[thm]{Lemma}
\newtheorem*{lemstar}{Lemma}
\newtheorem{conj}[thm]{Conjecture}
\newtheorem{ques}[thm]{Question}
\newtheorem*{conjstar}{Conjecture}
\theoremstyle{remark}
\newtheorem{rem}[thm]{Remark}
\newtheorem{np*}{Non-Proof}
\newtheorem*{remstar}{Remark}
\theoremstyle{definition}
\newtheorem{defn}[thm]{Definition}
\newtheorem*{defnstar}{Definition}
\newtheorem{exam}[thm]{Example}
\newtheorem*{examstar}{Example}
\newcommand{\pd}[2]{\frac{\partial #1}{\partial #2}}
\newcommand{\pdtwo}[2]{\frac{\partial^2 #1}{\partial #2^2}}
\def\Ind{\setbox0=\hbox{$x$}\kern\wd0\hbox to 0pt{\hss$\mid$\hss} \lower.9\ht0\hbox to 0pt{\hss$\smile$\hss}\kern\wd0}
\def\Notind{\setbox0=\hbox{$x$}\kern\wd0\hbox to 0pt{\mathchardef \nn=12854\hss$\nn$\kern1.4\wd0\hss}\hbox to 0pt{\hss$\mid$\hss}\lower.9\ht0 \hbox to 0pt{\hss$\smile$\hss}\kern\wd0}
\def\ind{\mathop{\mathpalette\Ind{}}}
\def\nind{\mathop{\mathpalette\Notind{}}} 
\newcommand{\m}{\mathbb }
\newcommand{\mc}{\mathcal }
\newcommand{\mf}{\mathfrak }

\title{Completeness in Partial Differential Fields}
\author{James Freitag}

\address{freitagj@gmail.com \\
Department of Mathematics, Statistics, and Computer Science\\
University of Illinois at Chicago\\
322 Science and Engineering Offices (M/C 249)\\
851 S. Morgan Street\\
Chicago, IL 60607-7045 }
\begin{abstract} We study completeness in partial differential varieties. We generalize many of the results of Pong to the partial differential setting. In particular, we establish a valuative criterion for differential completeness and use it to give a new class of examples of complete partial differential varieties. In addition to completeness we prove some new embedding theorems for differential algebraic varieties. We use methods from both differential algebra and model theory.
\end{abstract}

\maketitle

Completeness is a fundamental notion in algebraic geometry. In this paper, we examine an analogue in differential algebraic geometry. The paper builds on Pong's \cite{PongDiffComplete2000} and Kolchin's \cite{KolchinDiffComp} work on differential completeness in the case of differential varieties over ordinary differential fields and generalizes to the case differential varieties over partial differential fields with finitely many commuting derivations. Many of the proofs generalize in the straightforward manner, given that one sets up the correct definitions and attempts to prove the correct analogues of Pong's or Kolchin's results. Of course, some of the results are harder to prove because our varieties may be infinite differential transcendence degree. Nevertheless, Proposition \ref{proj} generalizes a theorem of \cite{PongRank} even when we restrict to the ordinary case. The proposition also generalizes a known result projective varieties. Many of the results are model-theoretic in nature or use model-theoretic tools.

The model theory of partial differential fields was developed in \cite{McGrail}. For a recent alternate (geometric) axiomatization of partial differentially closed fields, see \cite{Omar}. For a reference in differential algebra, we suggest \cite{KolchinDAAG} and \cite{MMP}. The differential varieties we consider will be embedded in projective space. Generalizations to differential schemes are of interest, but are not treated here. Pillay also considers differential completeness for a slightly different category in \cite{PillayDvar}. Though Pillay's conditions for differential completeness are implied by the conditions here, their precise relationship is not clear. Generalizing this work to the difference-differential topology is of interest, but there are important model theoretic obstacles. Specifically, when working in the setting of \cite{Medina}, we would not have quantifier elimination. We would also leave the model theoretic setting of superstability for the more general setting of supersimplicity. That category is of particular interest because we know that the natural analogues of questions \ref{Q1} and \ref{Q2} are actually distinct. 

In this paper, $F \models DCF_{0,m}.$ $C_F$ will denote the field of absolute constants of $F,$ that is, the intersection of the kernels the derivations $\delta \in \Delta.$ We will only consider differential fields in which the derivations commute. Throughout, $\mc U$ will be a very saturated $\Delta$-closed field. By using model-theoretic tools, Pong proved that every complete differential variety (over an ordinary differential field) is isomorphic to a complete subset of $\m A^1.$ We prove this result in the partial case, and give a valuative criterion for $\Delta$-completeness. More specifically, in section one, we set up the basic definitions and establish some simple propositions. Section two is devoted to an example of Ellis Kolchin, which is used later as well. Section three is devoted to some embedding theorems for differential algebraic varieties. In section four we establish the valuative criterion mentioned above. In section five we use this criterion to establish a wide class of new examples.

As $DCF_{0,m}$ has \emph{quantifier elimination} we have a correspondence between definable sets, (generic) tuples in field extensions lying on the variety, and types. Given a type $p \in S(K),$ we have a corresponding differential radical ideal via
$$p \mapsto I_p = \{ f \in K\{y \} \, | \, f(y)=0 \in p \}$$
Of course, the corresponding variety is simply the zero set of $I_p.$ So, when considering model theoretic ranks on types like Lascar rank (denoted $RU(p)$) we will write $RU(V)$ for a definable set (whose Kolchin closure is irreducible) for $RU(p)$ where $p$ is a type of maximal Lascar rank in $V$. We should note that some care is required, since the model theoretic ranks are always invariant under taking Kolchin closure. See \cite{FGenerics} for an example. Of course, sometimes it will be convenient to consider tuples in field extensions, in which case our notation using also uses the correspondence between definable sets, types, and tuples in field extensions. 

\section{Definitions}
A subset of $\m A^n$ is \emph{$\Delta$-closed} if it is the zero set of a collection of $\Delta$-polynomials in $n$ variables. We use $F \{ y_1 , \ldots , y_n \}$ to denote the ring of \emph{$\Delta$-polynomials} over $F$ in $y_, \ldots , y_n.$  For a thorough development, see \cite{KolchinDAAG} or \cite{MMP}. A (non-constant) $\Delta$-polynomial in $F\{ y_0 , \ldots , y_n \}$ is \emph{$\Delta$-homogeneous of degree d} if 
$$f(t y_0, \ldots t y_n \} = t^d f(y_0 , \ldots , y_n ),$$ 
for a $\Delta$-transcendental $t$ over $F \{y_0 , \ldots , y_n \}.$ One can easily homogenize an arbitrary $\Delta$-polynomial. Homogenization in the partial differential case works identically to the ordinary case. For details and examples see \cite{PongDiffComplete2000}. 

$\Delta$-closed subsets of $\m P^n$ are the zero sets of collections of homogeneous differential polynomials in $F \{y_0 , \ldots , y_n \}.$ $\Delta$-closed subsets of $\m P^n  \times \m A^m$ are given by the zero sets of collections of differential polynomials in $F \{ y_0 , \ldots , y_n, z_1, \ldots , z_m \}$ which are homogeneous in $\bar y.$ Usually we will consider \emph{irreducible} $\Delta$-closed sets, that is, those which are not the union of finitely many proper $\Delta$-closed subsets.

\begin{defn}
A $\Delta$-closed $X \subseteq \m P^n$ is $\Delta$-complete if the second projection $\pi_2: X \times Y \rightarrow Y$ is a $\Delta$-closed map for every quasi projective $\Delta$-variety, $Y.$ 
\end{defn}
It is not a restriction to consider only irreducible $\Delta$-closed sets. To see this, note that in proving that $X \times Y \rightarrow Y$ is $\Delta$-closed, it is enough to prove that the map is $\Delta$-closed on each irreducible component of $X.$ 

\begin{prop}\label{basics} If $X$ is $\Delta$-complete and $Y$ is a quasi projective $\Delta$-variety, 
\begin{enumerate}
\item Suppose $f: X \rightarrow Y$ continuous. Then $f(X)$ is $\Delta$-closed in $Y$ and $\Delta$-complete. 
\item Any $\Delta$-closed subset of $X$ is $\Delta$-complete. 
\item Suppose that $Y$ is $\Delta$-complete. Then $X \times Y$ is $\Delta$-complete. 
\end{enumerate}
\end{prop}
\begin{proof} 
Let $f: X \rightarrow Y \subseteq \m P^m.$ Then $f \times id: X \times \m P^m \rightarrow \m P^m \times \m P^m$ is a continuous map. The diagonal of $\m P^m \times \m P^m$ is $\Delta$-closed. By virtue of the completeness of $X,$ $\pi_2(gr(f))$ is $\Delta$-closed. This is $f(X).$ Now, we get the following commuting diagram, giving the completeness of the image, $f(X).$   
$$
\xymatrix{
& X \times Y \ar[r]^{f \times id} \ar[d]_{\pi_2} &  f(X) \times Y \ar[ld]^{\pi_2} & \\
 &   Y & &  
}
$$
For 2), simply note that if $Z$ is and $\Delta$-closed subset of $X,$ then we have the natural injective map $Z \times Y \rightarrow X \times Y.$ Further, the $\pi_2$ projection map clearly factors through this map. So, $Z$ must be $\Delta$-complete. Similarly, for 3), we can simply note that if we have $X \times Y \times Z,$ then the projection $X \times Y \times Z \rightarrow Y \times Z$ is closed by the completeness of $X.$ $Y \times Z \rightarrow Z$ is closed by the completeness of $Y$. The composition of closed maps is closed. 
\end{proof}

When we wish to verify that a differential algebraic variety, $X$, is complete, we need only show that $$\pi : X \times Y \rightarrow Y$$ is $\Delta$-closed for affine $Y.$ This fact is true for the same reason as in the case of algebraic varieties. $\Delta$-closedness is a local condition, so one should cover $Y$ by finitely many copies of $\m A^m$ for some $m$ and verify the condition on each of the affine pieces. 


\section{$\m P^n$ is not $\Delta$-complete}

There are more closed sets in Kolchin topology than in the Zariski topology, so in the differential setting, there are more closed sets in both the image spaces of the projection map (which makes completeness easier to achieve) and more closed sets in the domain (making completeness harder to achieve).  Though the $\Delta$-topology is richer than the Kolchin topology for ordinary differential fields, Pong's exposition \cite{PongDiffComplete2000} of Kolchin's theorem that $\m P^n$ is not complete holds in this setting. The techniques are model-theoretic, as in \cite{PongDiffComplete2000}, but use the model theory of $\Delta$-fields. 

Consider, for some $\delta \in \Delta$, the closed set in $\m P^1 \times \m A^1$ given by solutions to the equations
\begin{eqnarray}
z (\delta y)^2 + y^4 -1 =0 \\
2 z \delta^2 y + \delta z \delta y +4y^3 =0
\end{eqnarray}
Note that the projective closure with respect to the $y$ variable, does not contain the point at infinity. So, we can work with the above local equations. Let $b \in \mc U$ be a $\Delta$-transcendental over $F.$ Then, we have a solution, in $\mc U,$ to the equation $$b (\delta y)^2 + y^4 -1 =0$$ 
Further, we can demand that $4y^3-1 \neq 0.$ Call this solution $a.$  This means that $\delta a \neq 0.$ But, since we chose $b$ to be a $\Delta$-transcendental over $F,$ we know, by quantifier elimination, that if $\pi_2 Z$ is $\Delta$-closed, then it is all of $\m A^1.$ But $\pi_2 Z$ can not be all of $\m A^1,$ since $\pi_2 Z$ can not contain $0.$ Thus, $\pi_2 Z$ is not $\Delta$-closed and $\m P^1$ is not $\Delta$-complete. Since $\m P^1$ is a $\Delta$-closed subset of $m P^n$ for any $n,$ we have the following result of Kolchin,  

\begin{prop} $\m P^n$ is not $\Delta$-complete. 
\end{prop}

There are infinite rank definable subfields in partial differential fields (they are all given as the constant field of a set of independent definable derivations, see \cite{SuerThesis}). Suppose the set of definable derivations of cutting out the constant field, $K_0$,  is of size $m-m_1.$ $\m P^n (K_0)$ is not $\Delta$-complete. To see this, one can simply repeat the above techniques in a model of $DCF_{0,m_1}$ for $m_1 <m$.  

\section{Embedding Theorems}

Next, we will use the model-theoretic tool of Lascar rank to show that every $\Delta$-complete set can be embedded in $\m A^1.$ We may identify hypersurfaces of degree d in $\mathbb{P}^n$ with points in $\mathbb{P}^{\binom{d+n}{d}-1}$ via the following correspondence, $$\bar a=[a_1 \ldots a_{\binom{d+n}{d}}] \in \mathbb{P}^{\binom{d+n}{d}-1} \leftrightarrow  H_{\bar a}=Z(\sum_{i=1}^{\binom{d+n}{d}}a_i M_i )$$ where $M_i$ is the ith monomial in $x_0 \ldots x_n$ ordered lexicographically. Notice the unusual numbering on the parameter space of $a_i's.$ This follows Hartshorne \cite{Hartshorne}. We say that a hypersurface is generic if the associated $\bar a$ is a tuple of differential transcendentals over $F.$ This is equivalent to saying that $RU(a/F)=\omega^m \cdot N.$
\begin{prop}\label{Avoid}
Suppose $X \subset \mathbb{P}^n$ is a definable set of Lascar rank less than $\omega^m.$ Then any generic hypersurface does not intersect X. 
\end{prop}
\begin{proof}
Let $$H=Z( \sum_{i=1}^{\binom{d+n}{d}}a_i M_i )$$ where $\bar a$ is a generic point in $\mathbb{P}^N.$ Now, suppose that $x \in X \cap H.$ The $x$ specifies a proper subvariety of $\mathbb{P}^N,$ namely the hyperplane given by $\sum a_i M_i(x)=0$ This hyperplane is, of course, isomorphic to $\mathbb{P}^{N-1}.$ But, we know that the rank of a generic point in $\m P^{N-1}$ is not $\omega^m \cdot N,$ namely we know $RU(a/x)\leq \omega^m \cdot (N-1).$ Now, using the Lascar inequality, $$RU(a) \leq RU(a,x) \leq RU(a/x) \oplus RU(x) <\omega^m \cdot N.$$ 
So, $\bar a$ must not have been generic in $\mathbb{P}^N.$
\end{proof}

\begin{prop}\label{proj1}
Let $X \subseteq \m P^n$ be a proper $\Delta$-closed subset. Let $p \in \m P^n-X$ be generic. Suppose $RU(X) \geq \omega^m.$ Then $RU( \pi_p(X)) \geq \omega^m,$ where $\pi_p$ is projection from the point to any hyperplane not containing the point. 
\end{prop}
\begin{proof} Take $b \in X$ generic (of full $RU$-rank). Then let $c=\pi_p(b).$ Since we assumed that $p \in \m P^n -X,$ we know that the intersection of the line $\bar{pb}$ and $X$ is a proper Kolchin closed subset of the affine line. Thus, $RU(b/cp)<\omega^m.$ But then $$\omega^m \leq RU(b)=RU(b/p) \leq RU(b,c/p) \leq RU(b/p) \oplus RU(c/bp).$$ Of course, this implies that $RU(c) \geq \omega^m.$ 
\end{proof}

\begin{prop} \label{proj} Let $X$ be a $\Delta$-variety with Lascar rank less than $\omega^m (k+1) .$ Then $X$ is isomorphic to a definable subset of $\m P^{2k+1}.$ 
\end{prop}
\begin{proof} Suppose that $X \subseteq \m P^n.$ Let $p$ be a generic point of $\m P^n.$ Let $H$ be any hyperplane not containing $p.$ We claim that projection from $p$ to $H$, restricted to $X$, is an injective map. Suppose not. Then there are two points $X_1, x_2$ on $X$, so that $p$ is on the line joining $x_1$ and $x_2.$ Then $$RU(p) \leq  RU( p, x_1, x_2) \leq RU(p / x_1 x_2) \oplus RU(x_1 x_1) < \omega^m \cdot (2k+2)$$ This is a contradiction unless $n<2k+1.$ Iteratively projecting from a generic point gives the result. 
\end{proof}
\begin{rem}
In the special case that $X$ is an algebraic variety, then this simply says that we can construct a definable isomorphism to some constructible set in $\m P^{2 dim(X) +1}.$ 
\end{rem} 
From the previous two propositions, we get,
\begin{corr} Suppose that $X$ is of Lascar rank less than $\omega^m.$ Then $X$ is definably isomorphic to a definable subset of $\m A^1.$ 
\end{corr}
\begin{proof} Using Proposition \ref{proj}, we get an embedding of $X$ into $\m P^1.$ We know that $X$ avoids any generic point of $\m P^1$ by Proposition \ref{Avoid}. 
\end{proof}
\begin{rem} The use of Proposition \ref{Avoid} in the above proof is gratuitous, since it is clear that the projection of $X$ is a proper subset of $\m P^1,$ by simple rank computations. The proposition is less obvious when the variety is embedded in higher projective spaces. 
\end{rem}

\begin{thm} Any $\Delta$-complete set is of $RU$-rank strictly less than $\omega^m.$ 
\end{thm}
\begin{proof} 
Suppose $X$ is complete and of rank larger than $\omega^m.$ Projection from any generic point gives a $\Delta$-complete, $\Delta$-closed set (by Proposition \ref{basics}) of rank at least $\omega^m$ (by Proposition \ref{proj1}). Iterating the process yields such a set in $\m P^1.$ The only $\Delta$-closed subset of rank $\omega^m$ in $\m P^1$ is all of $\m P^1.$ This is a contradiction since $\m P^1$ is not $\Delta$-complete. 
\end{proof}

\begin{corr} Every $\Delta$-complete subset of $\m P^n$ is isomorphic to a $\Delta$-closed subset of $\m A^1.$ 
\end{corr}

\begin{rem} More results along the lines of those in \cite{PongThesis} computing bounds on ranks of various algebraic geometric constructions on differential varieties are certainly possible, but the above results suffice our puposes here.
\end{rem}
The following two questions may or may not be distinct. For discussion of this see \cite{Jindecomposability} and \cite{SuerThesis}.
\begin{ques} \label{Q1}
Are there infinite Lascar rank $\Delta$-complete sets?
\end{ques}
\begin{ques} \label{Q2}
Are there infinite transcendence degree $\Delta$-complete sets?
\end{ques}

\section{A Valuative Criterion for $\Delta$-completeness}

The following is a proposition of van Den Dries \cite{vandenDries}, which Pong used in the case of ordinary differential fields to establish a valuative criterion for completeness. We take a similar approach here. 

\begin{prop}\label{Lou} $T$ a complete $\mc L$-theory and $\phi(v_1 \ldots v_n)$ an $\mc L$-formula without parameters. Then the following are equivalent:
\begin{enumerate}
\item There is a positive quantifier free formula $\psi$ such that $T$ proves $\forall v \phi (v) \leftrightarrow \psi(v).$
\item For any $K, L \models T$ and $f: A \rightarrow L$ an embedding of a substructure $A$ of $K$ into $ L,$ if $a \in \m A^m$ and $K \models \phi(a)$ then $L \models \phi(f(a)).$ 
\end{enumerate} 
\end{prop}

\begin{prop} Let $R \supseteq \m Q$ be a $\Delta$-ring. Every $\Delta$-ideal $\m R$ is contained in some prime $\Delta$-ideal. 
\end{prop}
\begin{proof} Let $ I \subseteq R$ be a $\Delta$-ideal. There is a maximal element (with respect to containment) among the $\Delta$-ideals containing $I.$ In characteristic 0, radicals of differential ideals are differential ideals. Every radical $\Delta$-ideal is the intersection of finitely many prime $\Delta$-ideals. 
\end{proof}

The next definition is essential for the criterion we give for completeness.
\begin{defn} 
Let $K$ be a $\Delta$-field. $$H_K:= \{(A,f, L) : A \text{ is a } \Delta \text{-subring of } K, L \text{ is a } \Delta \text{-field}, f: A \rightarrow L \text{ a } \Delta \text{-homomorphism} \}.$$   
Given $(A_1, f_1 , L_1), (A_2, f_2 , L_2) \in H_K.$ Then $f_2$ extends $f_1$ if $ A_1 \subset A_2,$ $L_1 \subseteq L_2,$ and $ f_2 |_{A_1}=f_1.$ We denote this by writing $(A_1, f_1 , L_1) \leq (A_2, f_2, L_2).$ With respect to this ordering, $H_K$ has maximal elements. These will be called \emph{maximal $\Delta$-homomorphisms of $K.$}
\end{defn}

\begin{defn} A $\Delta$-subring is maximal if it is the domain of a maximal $\Delta$-homomorphism. A $\Delta$-ring is called a local $\Delta$-ring if it is local and the maximal ideal is a $\Delta$-ideal. 
\end{defn}

\begin{prop}\label{algebra} Let $(R,f,L)$ be maximal in $H_K.$ Then, 
\begin{enumerate}
\item $R$ is a local $\Delta$-ring and $ker(f)$ is the maximal ideal.
\item $x \in K-R$ if and only if $\mf m\{x \} = R\{x \}.$ 
\end{enumerate}
\end{prop}
\begin{proof}
Kernels of $\Delta$-homomorphisms are $\Delta$-ideals. Let $x \notin ker(f).$ Then we extend $f$ to $R_{(x)}$ by letting $x^{-1} \mapsto f(x)^{-1}.$ By maximality, $x^{-1} \in R.$ Thus, $x$ is a unit.\\
If $\mf m \{x \} = R\{x \}$ then $1 = \sum_{i=1}^k m_i p_i(x).$ Then if $x \in R,$ $f(1)= f(\sum m_ip_i(x))=0.$ 
If $\mf m \{x \} \neq R \{x \},$ then there is a prime $\Delta$-ideal $\mf m'$ which contains $\mf m \{x \}.$ So, we let $K'$ be the fraction field of $R \{x \} /\mf m'$. 
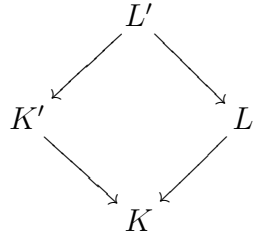
$$
\xymatrix{
& L' \ar[rd] \ar[ld] &  \\
 K' \ar[rd] & & L \ar[ld] \\
&  K &  
}$$
 Further, we have: 
$$
\xymatrix{
R \{x\} \ar[r] & K' \ar[r] & L'  \\
R \ar[r] \ar[u] & K \ar[r] \ar[u] & L \ar[u]
}$$

But, one can see that the maximality condition on $R$ means that $x \in R.$ 
\end{proof}

\begin{thm}
Let $X$ be a $\Delta$-closed subset of $\m A^n$ Then $X$ is $\Delta$-complete if and only if for any $K \models T$ and any $R,$ a maximal $\Delta$-subring of $K$ containing $F,$ we have that every $K$-rational point of $X$ is and $R$-rational point of $X$. 
\end{thm}
\begin{proof} Suppose the condition on maximal $\Delta$-subrings holds for $X$. Then, to show that $X$ is $\Delta$-complete, it suffices to show that for any $\Delta$-closed set $Z \subseteq X \times \m A^n,$ $\pi_2(Z)$ is $\Delta$-closed. Given $K$ and $f$ with $K,L \models DCF_{0,m},$ let $f: A \rightarrow L$ be a $\Delta$-homomorphism where $A$ is a substructure of $K.$ We let $\phi(y)$ be the formula saying $y \in \pi_2 Z.$ We will show that if $a \in \m A^n$ and $K \models \phi(a)$ then $L \models \phi(f(a)).$ So, we assume there is a $x \in X(K)$ with $(x,a) \in Z.$ Now, extend $f$ to $\tilde f: R \rightarrow L'$ a maximal $\Delta$-homomorphism. One can always assume that $L' \models DCF_{0,m},$ since if this does not hold, we simply take the $\Delta$-closure. At this point, we have $F \subseteq A \subseteq R.$ By the assumption, we know already that $x \in R.$ 
$$L' \models  (\tilde f(x), \tilde f(a)) \in Z \wedge \tilde f(x) \in X.$$ 
So, $L' \models \tilde f(a) \in \pi_2Z.$ But, then $L \models f(a) \in \pi_2 Z,$ since $\tilde f(a)= f(a).$  Now, using van Den Dries' condition, \ref{Lou}, we can see that $\pi_2 Z$ is $\Delta$-closed. \\

Now, we suppose that the valuative criterion does not hold of $X.$ So, we have some $f : R \rightarrow L$ a maximal $\Delta$-homomorphism of $K,$ with $R \supseteq F.$ There is some point $x \in X(K),$ so that $x \notin R.$ Then for one of the elements $x_i$ in the tuple $x,$ we know, by \ref{algebra} that $1 \in \mf m \{x_i \}.$ Then, we know that there are $m_j \in \mf m$ and $t_j \in R\{y\}$ so that $$\sum_j m_j t_j(x) =1.$$ We let $m:=(m_1 \ldots m_k).$  So, we let $\phi(y,z)$ be the formula which says: 
$$\sum_j z_j t_j(y_i)-1=0 \wedge y \in X.$$
Then we take $L$ to be the $\Delta$ closure of the $\Delta$-field $R/ \mf m.$ If we let $g$ be the quotient map then $g|_F$ is an embedding. Then we have that $K \models \exists y \phi(y, m)$ since $x$ is a witness. But, we see that $L$ can not have a witness to satisfy this formula, $m \in ker(f).$ Then again, by \ref{Lou}, we see that $\pi_2 Z$ is not $\Delta$-closed.  
\end{proof}

One can rephrase the criterion for an affine $\Delta$-closed subset $X$, a fact noticed by Pong, in the ordinary case. For any $K \models DCF_{0,m}$, let $R$ be a maximal $\Delta$-subring of $K$ which contains $F.$ Let $A = K\{y_1 , \ldots , y_n \} / I(X).$ Suppose we are given a commutative diagram, 
$$
\xymatrix{
 Spec K \ar[d] \ar[r] & Spec A \ar[d] \\
 Spec R \ar[r] & Spec F\\
}
$$
then we have the diagonal morphism, 
$$
\xymatrix{
 Spec K \ar[d] \ar[r] & Spec A \ar[d] \\
 Spec R \ar[r] \ar@{.>}[ur] & Spec F\\
}
$$
\section{New $\Delta$-complete varieties}
In this section, we will use some commutative algebra along with the valuative given above to give examples of $\Delta$-complete sets.
\begin{thm}
Let $V$ be the $\Delta$-closure of the zero set of $\{ \delta_i y - P_i(y) \} _{i=1} ^m $ in $\m P^1.$ Then $V$ is $\Delta$-complete.
\end{thm}
\begin{proof}
First, if some of the $P_i$ are linear, we should note that by  an appropriate linear translation of the field of constants of $\delta_i$ that the theorem reduces to the analogous theorem in $DCF_{0,m_1}$ for $m_1 <m.$ This follows because in a partial differentially closed field, any definable set is stably embedded (for a general discussion, see \cite{GST}). So, definability in the field of constants of any derivation is the same a definability in the structure with the language missing that derivation. Completeness of the projective closure of the constant field $C_F$ (indeed, completeness of the $C_F$ rational points of any complete variety, in the algebraic category), follows by the completeness of projective space in the algebraic category. So, for the rest of the proof, assume that at least one of the polynomials $P_i$, is of degree greater than 1.
 
One should homogenize $\delta_i y -P_i y$ in $\m P^1$ to calculate the $\Delta$-closure. If $deg P_i$ is $d_i$ then $H_i ( y, y_1) = y_1 ^d (\delta_i(\frac{y}{y_1})- P_i(\frac{y}{y_1}))$ is homogeneous. One can easily see (by examining the leading monomial of $P_i$) that $[1,0] \notin Z(H_i(y,y_1)).$ Thus, the set $V$ in affine space is equal to its projective closure. Now, we can use the valuative criteria to establish the completeness of the variety. \\

Let $x \in V(K)$ and $R$ a maximal $\Delta$-subring containing $F.$ It is enough to show that $\mf m_K \{x \} \neq R \{x \}.$ Since $\delta_i x = P_i(x),$  $$\mf m\{ x \} = \mf m[x]$$
$$R \{x \} = R[x].$$ 
So, since by classical commutative algebra, either $\mf m[x]$ or $\mf m[x^{-1}]$ is the unit ideal, see \cite{AtiyahMac}, it is enough to show that $\mf m[x^{-1}]$ is the unit ideal. Now, we take the approach of Pong \cite{PongDiffComplete2000}, using a result of Blum \cite{BlumExtensions} which is also contained in \cite{MorrisonSD}. 

\begin{defn} Let $R$ be a $\delta$ ring. An element of $R \{y \} $ is monic if it is of the form $y^n+ f(y)$ where the total degree of $f$ is less than $n.$ An element in a $\delta$-field extension of $R$ is monic over $R$ if it is the zero of a monic $\delta$-polynomial. 
\end{defn}

\begin{prop} If $(R, \mf m)$ is a maximal $\delta$-subring of $K,$ then $x \in K$ is monic over $R$ if and only if $x^{-1} \in \mf m.$ 
\end{prop}
\begin{rem} Though the results of Blum and Morrison are in the context of ordinary differential algebra, Proposition \ref{algebra} lets us apply their results, since a local $\Delta$-ring is a local $\delta$-ring for any $\delta \in \Delta.$ 
\end{rem}

Since $deg P_i \geq 2,$ $$a^{-1} (\delta_i y - P_i(y))$$ is monic, where $a$ is the coefficient of the leading monomial of $P_i.$ $x$ monic implies $x^{-1} \notin \mf m$,  by Blum's theorem \cite{Blum}.

So, if $x^{-1} \in R,$ then $x \in R.$ Thus, we assume $x^{-1} \notin R.$ Then $\mf m \{ x^{-1} \}$ is the unit ideal. $$\delta_i(x^{-1})= -x^{-2} P_i(x).$$
This means that we can get some expression, \begin{eqnarray}\label{eqninduct} 1= \sum_{j=r}^s m_j x^j \end{eqnarray} with $m_j \in \mf m$ and $r$ and $s$ integers. 
Then, applying $\delta_i$ to both sides of the equation yields: 
$$1= \sum_{j=r}^s \delta_i (w_j) x^j+ \pd{x^j}{x}P_i(x)$$
The leading term of the sum is $a m_s s x^{s-1} x^{d}.$  But, $a$ is a unit and so we can divide both sides of the equation by $sa x^{d-1}$ to obtain an expression for $m_s x^s$ as a sum of lower degree terms. Substituting this expression into \ref{eqninduct}, we get an expression \begin{eqnarray} 1= \sum_{j=r_1}^{s_1} n_j x^j \end{eqnarray}
Continuing in this manner, one can assume that we have an expression of the form 
$$1=\sum _{j=r}^0 w_j x^j.$$ 
Then we see that $1 \in \mf m[x^{-1}].$
\end{proof}

\begin{rem} Beyond order 1 the techniques as  shown above are not as easy to apply. For techniques in that situation, at least in the case of linear ordinary differential varieties, see \cite{SimmonsLinear}. One can combine the above techniques of that paper with the above techniques to give a wider class of complete partial differential varieties.
\end{rem}

\section{Acknowledgements}
I would like to thank Dave Marker and William Simmons for many helpful discussions on this topic. I would also like to thank the participants of the Kolchin seminar and Alexey Ovchinnikov in particular for useful discussions and encouragement.

\bibliography{Research}{}
\bibliographystyle{plain}

\end{document}